\documentclass[12pt]{article} 

\usepackage{amssymb}
\usepackage{amsmath} 
\usepackage{eucal}
\usepackage{amsfonts}
\usepackage{amsthm}
\usepackage[colorlinks]{hyperref} 

\newtheorem{thm}{Theorem}[section]
\newtheorem{cor}[thm]{Corollary}
\newtheorem{lem}[thm]{Lemma}
\newtheorem{prop}[thm]{Proposition}

\theoremstyle{definition}
\newtheorem{defn}[thm]{Definition}
\newtheorem{ex}[thm]{Example}

\newcommand{\HH}{\mathcal{H}}
\newcommand{\K}{\mathcal{K}}

\newcommand{\E}{\mathcal{E}}

\newcommand{\R}{\mathbb{R}}
\newcommand{\C}{\mathbb{C}}
\newcommand{\N}{\mathbb{N}}

\newcommand{\leqs}{\leqslant}
\newcommand{\geqs}{\geqslant}

\newcommand{\ap}{\alpha}
\newcommand{\bt}{\beta}

\newcommand{\ep}{\epsilon}
\newcommand{\ld }{\lambda}

\newcommand{\norm}[1]{\lVert #1 \rVert}

\begin{document}

\centerline {\Large{\bf Operator Monotone Functions: Characterizations}}

\centerline{}

\centerline{\Large{\bf and Integral Representations}}

\centerline{}

\centerline{\bf {Pattrawut Chansangiam
\footnote{Email: kcpattra@kmitl.ac.th, Tel. +66 84 0777581 \\
Address: Department of Mathematics, Faculty of Science, King Mongkut's Institute of Technology Ladkrabang,
Chalongkrung Rd., Ladkrabang, Bangkok 10520, THAILAND.}}}

\centerline{}



\noindent \textbf{Abstract:} Operator monotone functions, introduced by L\"{o}wner in 1934, are an important class of real-valued functions.
They arise naturally in matrix and operator theory and have various applications in other branches of mathematics
and related fields. This concept is closely related to operator convex/concave functions.
In this paper, we provide their important examples and characterizations in terms of matrix of divided differences.
Various characterizations and the relationship between operator monotonicity and operator convexity are given by Hansen-Pedersen characterizations.
Moreover, operator monotone functions
on the nonnegative reals have special properties, namely, they admit integral representations with respect to suitable
Borel measures. \\



\section{Introduction}

A useful and important class of real-valued functions is the class of operator monotone functions.
Such functions were introduced by L\"{o}wner in a seminal paper \cite{Lowner}.
These functions are functions of Hermitian matrices/operators preserving order.
In that paper, he established a relationship between operator monotonicity, the positivity of matrix of divided differences and an important class of analytic functions, namely, Pick functions.
This concept is closely related to operator convex/concave functions which was studied afterwards by Kraus in \cite{Kraus}.
Operator monotone functions and operator convex/concave functions
arise naturally in matrix and operator inequalities (e.g. \cite{Ando}, \cite{Bhatia},
\cite{Bhatia_positive def matrices}, \cite{Zhan}).
This is because the theory of inequalities depends heavily on the concepts of monotonicity, convexity and concavity.
One of the most beautiful and important results in operator theory is the so-called L\"{o}wner-Heinz
inequality (see \cite{Heinz}, \cite{Lowner}) which is equivalent to the operator monotonicity of the function $t \mapsto t^p$ for $t \geqs 0$ when $p \in [0,1]$.
Operator monotone functions have applications in many areas, including mathematical physics
and electrical engineering.
They arise in analysis of electrical networks (see e.g. \cite{Anderson-Trapp}).
They also occur in problems of elementary particles (see e.g. \cite{Wigner-von Neumann}).
Operator monotone functions play major roles in the so-called Kubo-Ando theory of operator connections and operator means.
This axiomatic theory was introduced in \cite{Kubo-Ando} and play important role in operator inequalities, operator equations, network theory and quantum information theory.
Indeed, there is a one-to-one correspondence between operator monotone functions
on the nonnegative reals $\R^+$ and operator connections.
See more information in \cite{Bhatia}, \cite{Donoghue}
and \cite{Hiai-Yanagi}.

In this paper, we survey significant results of operator monotone functions.
We give various characterizations in terms of matrix of divided differences.
Hansen-Pedersen characterizations provide characterizations and relationship of operator monotonicity and operator convexity.
Every operator monotone function on the nonnegative reals always occurs as an integral
of suitable operator monotone functions with respect to a Borel measure.
Such functions form building blocks for arbitrary operator monotone functions on $\R^+$.

Here is the outline of the paper. In Section 2, after setting basic notations, we give the definitions
and examples of operator monotone/convex functions and provide their characterizations with respect to
matrix of divided differences. Section 3 deals with Hansen-Pedersen characterizations of operator monotone/convex functions.  We consider operator monotone functions on the nonnegative reals in Section 4.

\section{Operator monotonicity and convexity}

Let $M_n$ be the algebra of $n \times n$ complex matrices.
The spectrum or the set of eigenvalues of $A \in M_n$ is denoted by $\sigma(A)$.
The set of $n \times n$ Hermitian matrices is written as $M_n^{sa}$.
The set of $n \times n$ positive semidefinite matrices is written by $M_n^+$.
The real vector space $M_n^{sa}$ is naturally equipped with a partial order as follows.
For $A,B \in M_n^{sa}$, define $A \leqs B$ if and only if $B-A$ belongs to its positive cone $ M_n^+$.
We write $A > 0$ to means that $A$ is a positive definite matrix, or equivalently,
$A \geqs 0$ and $A$ is invertible.

Let $A \in M_n$ be normal. Then there exist distinct scalars $\ld_1, \dots, \ld_m \in \C$
and projections $P_1, \dots, P_m$ on $\C^n$ such that
		\begin{align*}
				A = \sum_{i=1}^m \ld_i P_i, \quad P_i P_j = 0 \text{ for } i \neq j,
				\quad \sum_{i=1}^m P_i = I.
		\end{align*}
Moreover, these scalars and projections are uniquely determined.
In fact, $\sigma(A) = \{ \ld_1, \dots, \ld_m \}$ and each $P_i$ is the projection onto the eigenspace
$\ker (A - \ld_i I)$.
This decomposition is called the spectral resolution of $A$.
Consider a function $f: \sigma(A) \to \C$.
From the spectral resolution of $A$, we can define the functional calculus of the function $f$
by
	\begin{align*}
		f(A) = \sum_{i=1}^m f(\ld_i) P_i.
	\end{align*}
When $f$ is a polynomial, this definition coincides with the usual definition.
\begin{defn}
	Let $J \subseteq \R$ be an interval. A function $f: J \to \R$ is said to be
	\begin{itemize}
		\item	\emph{matrix monotone of degree $n$} or \emph{$n$-monotone} if, for every $A,B \in M_n^{sa}$ with
					$\sigma(A), \sigma (B) \subseteq J  $,
					\begin{align*}
						A \leqs B \quad \Longrightarrow \quad f(A) \leqs f(B).
					\end{align*}
		\item	\emph{operator monotone} or \emph{matrix monotone} if it is $n$-monotone for every $n \in \N$.
		\item	\emph{matrix convex of degree $n$} or \emph{$n$-convex} if, for every $A,B \in M_n^{sa}$ with
					$\sigma(A), \sigma (B) \subseteq J  $,
					\begin{align}
						f(tA + (1-t)B) \quad \leqs \quad tf(A) + (1-t)f(B). \label{eq: def of matrix convex}
					\end{align}
		\item	\emph{operator convex} if it is $n$-convex for every $n \in \N$.
		\item	\emph{matrix concave of degree $n$} or \emph{$n$-concave} if $-f$ is $n$-convex.
		\item	\emph{operator concave} if it is $n$-concave for every $n \in \N$.
	\end{itemize}
\end{defn}

Recall that a continuous function $f: J \to \R$ is convex (concave) if and only if it is midpoint-convex (midpoint-concave, respectively).
By passing this fact to the functional calculus, a continuous function
$f: J \to \R$ is $n$-convex if and only if it is $n$-midpoint convex,
i.e. \eqref{eq: def of matrix convex} holds for $t=1/2$.
In particular, if $f$ is continuous, then $f$ is operator convex if and only if it is operator midpoint-convex.
Analogous results are applied for the case of concavity.

Every $n$-monotone function is $(n-1)$-monotone but the converse is false in general.
The condition of being $1$-monotone is the monotone increasing in usual sense.
The set of operator monotone functions on $J$ is closed under taking nonnegative linear combinations,
pointwise limits and compositions.
The straight line $t \mapsto mt+c$ is operator concave and operator convex on the real
line for any $m,c \in \R$. This function is operator monotone if and only if the slope $m$ is nonnegative.

\begin{prop}	\label{Lemma 2.5.5} 
	On $(0,\infty)$, the function $t^{-1}$ is operator convex and $-t^{-1}$ is operator monotone.
	On $(-\infty, 0)$, the function $t^{-1}$ is operator concave and $-t^{-1}$ is operator monotone.
\end{prop}
\begin{proof}
	If $A \geqs B > 0$, then $A^{-1} \leqs B^{-1}$ and hence $-A^{-1} \geqs -B^{-1}$.
	The scalar inequality $[(1+t)/2] ^{-1} \leqs (1 + t^{-1})/2$ implies that for every $C>0$
	\begin{align*}
		\left( \frac{I+C}{2} \right)^{-1}  \leqs  \frac{I + C^{-1}}{2}.
	\end{align*}
	For $A,B>0$ in $M_n$, by setting $C = A^{-1/2} B A^{-1/2}$ we have
	\begin{align*}
			\left( \frac{A+B}{2} \right)^{-1}
				&= \left( \frac{A^{1/2} (I + C) A^{1/2}}{2} \right)^{-1} \\
				&=	A^{-1/2} \left(  \frac{I+C}{2} \right)^{-1} A^{-1/2} \\
				&\leqs A^{-1/2} \left( \frac{I + C^{-1}}{2} \right) A^{-1/2}
				=	\frac{A^{-1} + B^{-1}}{2}.
	\end{align*}
	Hence $t^{-1}$ is operator convex. For the case $(-\infty, 0)$, consider $-A$ and $-B$ instead of $A$ and $B$.
\end{proof}

It follows from this proposition that for $c \notin (a,b)$ the function $t \mapsto (c-t)^{-1}$ is operator monotone on $(a,b)$.
The next result is called the L\"{o}wner-Heinz inequality. It was first proved by L\"{o}wner \cite{Lowner} and
also  by Heinz \cite{Heinz}. There are many proofs of this fact. The following is due to Pedersen \cite{Pedersen}.

\begin{thm}
	For $A \geqs B$ in $M_n^+$ and $r \in [0,1]$, we have $A^r \geqs B^r$.
\end{thm}
\begin{proof}
	The continuity argument allows us to consider $A \geqs B >0$.
Since $p \mapsto A^p$ and $p \mapsto B^p$ are continuous, the set
\begin{align*}
	\triangle = \{p \in \R : A^p \geqs B^p\}
\end{align*}
is closed.
Clearly, $0,1 \in \triangle$.
Hence, to prove that $[0,1] \subseteq \triangle$, it suffices to show that
\begin{align*}
	p,q \in \triangle  \implies \frac{p+q}{2} \in \triangle.
\end{align*}
Here, we use the fact that the set of dyadic numbers in $[0,1]$ is dense in $[0,1]$.
	Suppose $A^p \geqs B^p$ and $A^q \geqs B^q$.
	Then $A^{-p/2} B^p A^{-p/2} \leqs I$ and
	\begin{align*}
		\norm{B^{p/2}A^{-p/2}}^2 = \norm{(B^{p/2}A^{-p/2})^* (B^{p/2} A^{-p/2})} =
        \norm{A^{-p/2} B^p A^{-p/2}} \leqs 1.
	\end{align*}
	Hence, $\norm{B^{p/2} A^{-p/2}} \leqs 1$ and similarly $\norm{B^{q/2} A^{-q/2}} \leqs 1$.
	Thus
	\begin{align*}
		1 &\geqs \norm{(B^{p/2}A^{-p/2})^* (B^{q/2} A^{-q/2})} = \norm{A^{-p/2} B^{(p+q)/2} A^{-q/2}} \\
			&\geqs r(A^{-p/2} B^{(p+q)/2} A^{-q/2}) \\
			&=		r(A^{-(p+q)/4} B^{(p+q)/2} A^{-(p+q)/4}) \\
			&=		\norm{A^{-(p+q)/4} B^{(p+q)/2} A^{-(p+q)/4}}.
	\end{align*}
    Here, $r(\cdot)$ denotes the spectral radius.
	Now, $I \geqs A^{-(p+q)/4} B^{(p+q)/2} A^{-(p+q)/4}$ or $A^{(p+q)/2} \geqs B^{(p+q)/2}$,
    i.e. $(p+q)/2 \in \triangle$.
\end{proof}

\begin{prop} 
	For each $p >1$, the function $t \mapsto t^p$ is not operator monotone on $\R^+$.
	\begin{proof}
	Consider
	$A = \begin{bmatrix} 	3/2 & 0 \\  0 & 3/4 \end{bmatrix}$
	and $B = \begin{bmatrix} 	1/2 & 1/2 \\  1/2 & 1/2 \end{bmatrix}$.
	Then $A \geqs B \geqs 0$.
	Since $B$ is a projection, for each $p>0$ we have $B^p = B$ and
	\begin{align*}
			A^p - B^p = \begin{bmatrix} 	(3/2)^p - 1/2 & -1/2 \\  -1/2 & (3/4)^p - 1/2 \end{bmatrix}.
	\end{align*}
	Compute
	\begin{align*}
		\det (A^p - B^p) = \left(\frac{3}{8} \right)^p \left(3^p - \frac{2^p+4^p}{2} \right).
	\end{align*}
	If $A^p \geqs B^p$, we must have $\det (A^p - B^p) \geqs 0$, i.e.
	$
		\frac{2^p + 4^p}{2} \leqs 3^p,
	$
	which is false when $p>1$.	
	\end{proof}
\end{prop}

\begin{thm} \label{Theorem 2.4.1}
    If $f$ is a $2$-monotone function on $(a,b)$, then $f$ is $C^1$ on $(a,b)$ and $f'>0$ unless $f$ is a constant.
    In particular, every operator monotone function on $(a,b)$ is $C^1$.
\end{thm}
\begin{proof}
    The proof is very long and it consists of many details.
    The original proof is contained in \cite{Lowner}; see also \cite{Donoghue}.
\end{proof}

\begin{thm}
    Let $n \geqs 2$ be an integer. The following statements are equivalent for a function $f: (a,b) \to \R$:
    \begin{enumerate}
        \item   $f$ is $n$-monotone on $(a,b)$;
        \item   $f$ is $C^1$ on $(a,b)$ and $[f^{[1]}(\ld_i, \ld_j)]_{i,j=1}^n \geqs 0$ for every choice of
                $\ld_1 < \ld_2 < \dots < \ld_n$ from $(a,b)$.
    \end{enumerate}
    	Here, the {1st divided difference} $f^{[1]} (x, y)$ is defined to be $\frac{f(x) - f(y)}{x - y}$
    	for $x \neq y$ and $f^{[1]}(x,x)= f'(x)$.
\end{thm}
\begin{proof}
    See \cite{Lowner}.
\end{proof}

\begin{thm} \label{thm 2.4.4 and cor 2.4.6}
    Let $n \geqs 2$ be an integer. The following statements are equivalent for a function $f: (a,b) \to \R$:
    \begin{enumerate}
        \item   $f$ is $n$-convex on $(a,b)$;
        \item   $f$ is $C^2$ on $(a,b)$ and $[f^{[2]}(\ld_1, \ld_i, \ld_j)]_{i,j=1}^n \geqs 0$ for every choice of
                $\ld_1, \ld_2, \dots, \ld_n$ from $(a,b)$.
    \end{enumerate}
    Moreover, if $f$ is operator convex, then $f^{[1]}(\ld, \cdot)$ is operator monotone for every $\ld \in (a,b)$.
    Here, the {2nd divided difference} $f^{[2]} (x,y,z)$ is defined to be
        $\frac{f^{[1]} (x, y) - f^{[1]} (y,z)}{ x - z}$
    	for $x \neq z$ and $f^{[2]}(x,y,x) = f''(x)$.
\end{thm}
\begin{proof}
    See \cite{Kraus}.
\end{proof}

Operator monotone functions can be defined in the context of operators acting on a Hilbert space as
illustrated in the next theorem.
This is why we also call a matrix monotone function an operator monotone function.
Note that in this theorem we assume the continuity of $f$ since we need to define the continuous functional
calculus of an operator. Here, $B(\HH)$ denotes the algebra of bounded linear operators on a Hilbert space $\HH$.

\begin{thm}
The following statements are equivalent for a continuous function $f: (a,b) \to \R$:
\begin{enumerate}
    \item[(i)]   $A \leqs B \implies f(A) \leqs f(B)$ for all Hermitian matrices $A,B$ of all orders whose spectrums are contained in $(a,b)$;
    \item[(ii)]   $A \leqs B \implies f(A) \leqs f(B)$ for all Hermitian operators $A,B \in B(\HH)$ whose spectrums
            are contained in $(a,b)$ and for an infinite-dimensional Hilbert space $\HH$;
    \item[(iii)]   $A \leqs B \implies f(A) \leqs f(B)$ for all Hermitian operators $A,B \in B(\HH)$ whose spectrums are contained in $(a,b)$ and for all Hilbert spaces $\HH$.
\end{enumerate}
\end{thm}
\begin{proof}
It is obvious that (iii) implies (ii). The implication (ii) $\Rightarrow$ (i) follows by taking
an $n$-dimensional subspace.

(i) $\Rightarrow$ (iii).
For each finite-dimensional subspace $F$ of $\HH$, let $P_F$ be the orthogonal projection onto $F$.
Suppose that $A \leqs B$ in $B(\HH)$ with spectra in $(a,b)$.
Consider nets $A_F:=P_F A P_F + c (I - P_F)$ and $B_F:=P_F B P_F + c(I - P_F)$ in
$B(\HH)$, where $c \in (a,b)$ is fixed and a directed set
\begin{equation*}
    \{F :  F \text{ is a finite-dimensional subspace of } \HH\}
\end{equation*}
with respect to the set inclusion.
Since $A_F \to A$ and $B_F \to B$ in the strong operator topology, we have $f(A_F) \to f(A)$ and
$f(B_F) \to f(B)$ in the strong operator topology. Note that $f(A_F) = f(P_F A P_F) + f(c) (I - P_F)$, where $f(P_F A
P_F)$ is the functional calculus of $P_F A P_F$ in $B(F)$. Since $B(F)$ is identified with $M_n$ with $n = \dim F$ and
since $P_F A P_F \le P_F B P_F$ as elements of $B(F)$, (iii) implies that $f(P_F A P_F) \le f(P_F B P_F)$ and hence
$f(A_F) \le f(B_F)$. By taking the limit in the strong operator topology, we have $f(A) \le f(B)$.
\end{proof}

\section{Hansen-Pedersen characterizations}

In this section, we characterize operator monotone functions in the sense of Hansen-Pedersen \cite{Hansen-Pedersen}.

\begin{lem} \label{Lemma 2.5.1}
	(1) Assume that $A \in M_n$ is normal and $U \in M_n$ is unitary.
	Then for every function $f$ on $\sigma (A)$,
			$f(U^* A U) = U^* f(A) U$.
			
	(2) For every $X \in M_n$ and every function $f$ on $\sigma(X^*X)$, we have
	$Xf(X^*X) = f(XX^*)X$.
\end{lem}
\begin{proof}
	(1)	Take the spectral resolution $A = \sum_{i = 1}^m \ap_i P_i$.
		Then $U^* A U = \sum_{i = 1}^m \ap_i U^* P_i U$
		is the spectral resolution of $U^*AU$.
		Hence
				\begin{align*}
					f(U^*AU) = \sum_{i = 1}^m f(\ap_i) U^* P_i U = U^* f(A) U.
				\end{align*}

	(2) Since $\sigma(X^*X) = \sigma(XX^*)$, $f(XX^*)$ is well-defined.
		Since $X(X^*X)^k = (XX^*)^k X$ for all $k \in \N$, the assertion holds
		when $f$ is a polynomial.
		Let $f$ be an arbitrary function on
		$\sigma(X^*X) = \{ \ap_1,\dots,\ap_m \}$.
		Define the {Lagrange interpolation polynomial}
		\begin{align*}
			p(t) = \sum_{j=1}^m f(\ap_j) \prod_{1 \leqs i \leqs m, i \neq j}
			\frac{t - \ap_j}{\ap_i - \ap_j},
		\end{align*}
		which is a polynomial such that $p(\ap_i) = f(\ap_i)$ for
		$1 \leqs i \leqs m$.
		It follows from (1) that $Xf(X^*X) = Xp(X^*X) = p(XX^*)X = f(XX^*)X$.
\end{proof}

\begin{thm} \label{thm 2.5.2}
	Let $f: [0,\ap) \to \R$ be a function where $0 < \ap \leqs \infty$. Then the following are equivalent:
	\begin{enumerate}
			\item[(i)]	$f$ is operator convex on $[0,\ap)$ and $f(0) \leqs 0$;
			\item[(ii)]	$f$ is operator convex on $(0,\ap)$ and $f(0^+) \leqs f(0)
						\leqs 0$, where the existence of $f(0^+):=\lim_{t \to 0^+} f(t)$ and $f(0^+) \leqs f(0)$
                        are automatic from the operator convexity of $f$ on $(0,\ap)$;
			\item[(iii)]	$f(t)/t$ is operator monotone on $(0,\ap)$ and
						$f(0^+) \leqs f(0) \leqs 0$, where the existence of $f(0^+)$ and $f(0^+) \leqs f(0)$
                        are automatic from the operator monotonicity of $f(t)/t$ on $(0,\ap)$;
			\item[(iv)]	$f(X^*AX) \leqs X^* f(A) X$ for every $A \in M_n^{sa}$
						with $\sigma (A) \subset [0,\ap)$, for every $X \in M_n$
						with $\norm{X} \leqs 1$ and for every $n \in \N$;
			\item[(v)]	$f(X^*AX + Y^*BY) \leqs X^*f(A)X + Y^* f(B)Y$ for every
						$A,B \in M_n^{sa}$ with
						$\sigma(A), \sigma(B) \subseteq [0,\ap)$,
						for every $X,Y \in M_n$ with $X^*X+Y^*Y \leqs I$ and
						for every $n \in \N$;
			\item[(vi)]	$f(PAP) \leqs P f(A) P$ for every $A \in M_n^{sa}$
						with $\sigma(A) \subseteq [0,\ap)$, for every orthogonal
						projection $P$ on $\C^n$ and for every $n \in \N$.
	\end{enumerate}
\end{thm}
\begin{proof}
	(i) $\Rightarrow$ (iv).
			For $A,X$ as in (iv), define $\tilde{A}, U,V \in M_{2n}(\C)$ by
			\begin{align*}
				\tilde{A} = \begin{bmatrix} A & 0 \\ 0 & 0 \end{bmatrix},
				U = \begin{bmatrix} X & R \\ (I-X^*X)^{1/2}  & -X^* \end{bmatrix}, 
				V = \begin{bmatrix} X & -R \\ -(I-X^*X)^{1/2}  & X^* \end{bmatrix}
			\end{align*}
			where $R:= (I-XX^*)^{1/2}$.
			Lemma \ref{Lemma 2.5.1}(2) implies $X(I-X^*X)^{1/2} = (I-XX^*)^{1/2}X$.
			Direct computations show that $U$ and $V$ are unitary.			
			Hence (i) and Lemma \ref{Lemma 2.5.1}(1) imply that
			\begin{align*}
				\begin{bmatrix}	f(X^*AX) & 0 \\ 0 & f(RAR) \end{bmatrix}
					&=	f(\begin{bmatrix}	X^*AX & 0 \\ 0 & RAR \end{bmatrix}) 
					=	f(\frac{U^* \tilde{A} U + V^* \tilde{A} V}{2})  \\
					&\leqs \frac{f(U^* \tilde{A} U) + f(V^* \tilde{A} V)}{2}
					= \frac{U^* f(\tilde{A}) U + V^* f(\tilde{A}) V}{2}\\
					&=	\frac{1}{2} U^* \begin{bmatrix}	f(A) & 0 \\ 0 & f(0)I \end{bmatrix} U
							+ \frac{1}{2} V^* \begin{bmatrix}	f(A) & 0 \\ 0 & f(0)I \end{bmatrix} V \\
					&\leqs \frac{1}{2} U^* \begin{bmatrix}	f(A) & 0 \\ 0 & 0 \end{bmatrix} U
							+ \frac{1}{2} V^* \begin{bmatrix}	f(A) & 0 \\ 0 & 0 \end{bmatrix} V \\
					&= \begin{bmatrix}	X^*f(A)X & 0 \\ 0 & Rf(A)R \end{bmatrix}.
			\end{align*}
			Thus $f(X^*AX) \leqs X^*f(A)X$.

	(iv) $\Rightarrow$ (v).
	For $A,B,X,Y$ as in (v) define
	\begin{align*}
				\tilde{A} = \begin{bmatrix} A & 0 \\ 0 & B \end{bmatrix},
				\tilde{X} = \begin{bmatrix} X & 0 \\ Y & 0 \end{bmatrix}.
			\end{align*}
	Since $\tilde{X}^*X = \begin{bmatrix} X^*X + Y^*Y & 0 \\ 0 & 0 \end{bmatrix}
	\leqs \begin{bmatrix} I & 0 \\ 0 & I \end{bmatrix}$,
	we have $\norm{\tilde{X}} \leqs 1$.
	Also $\tilde{A}^*=A$ and $\sigma(\tilde{A}) = \sigma(A) \cup \sigma(B)
	\subseteq [0,\ap)$.
	Since $\tilde{X}^*AX = \begin{bmatrix} X^*AX + Y^*AY & 0 \\ 0 & 0 \end{bmatrix}$,
	we get
	\begin{align*}
		\begin{bmatrix} f(X^*AX + Y^*BY) & 0 \\ 0 & f(0)I \end{bmatrix}
			&=	f(\tilde{X}^*\tilde{A} \tilde{X}) \leqs \tilde{X}^*f(\tilde{A}) \tilde{X}  \\
			&= \begin{bmatrix} X^*f(A)X + Y^*f(B)Y & 0 \\ 0 & 0 \end{bmatrix}
	\end{align*}
	and $f(X^*AX + Y^*BY) \leqs X^*f(A)X + Y^*f(B)Y$.
	
	(v) $\Rightarrow$ (vi). Put $X=P$ and $Y=0$ in (v).
	
	(vi) $\Rightarrow$ (i).
	For $A,B \in M_n^{sa}$ with $\sigma(A), \sigma(B) \subseteq [0,\ap)$
	and $0<\ld<1$, define
	\begin{align*}
			\tilde{A} = \begin{bmatrix} A & 0 \\ 0 & B \end{bmatrix},
			U = \begin{bmatrix} \ld^{1/2}I & -(1-\ld)^{1/2}I \\ (1-\ld)^{1/2}I & \ld^{1/2}I \end{bmatrix},
			P = \begin{bmatrix} I & 0 \\ 0 & 0 \end{bmatrix}.
	\end{align*}
	Then $\tilde{A}^* = \tilde{A}$ with $\sigma(\tilde{A}) \subseteq [0,\ap)$, $U$ is unitary and $P$ is a projection.
	Now,
	\begin{align*}
		PU^*\tilde{A}UP = \begin{bmatrix} \ld A + (1-\ld)B & 0 \\ 0 & 0 \end{bmatrix}.
	\end{align*}
	Hence (vi) implies that
	\begin{align*}
		\begin{bmatrix} f(\ld A + (1-\ld)B) & 0 \\ 0 & f(0)I \end{bmatrix}
			&=	f(PU^*\tilde{A}UP) \\
			&\leqs	P f(U^* \tilde{A} U) P = P U^* f(\tilde{A}) UP \\
			&=	 \begin{bmatrix} \ld f(A) + (1-\ld)f(B) & 0 \\ 0 & 0 \end{bmatrix}.
	\end{align*}
	Thus $f(\ld A + (1-\ld)B) \leqs \ld f(A) + (1-\ld)f(B)$	and $f(0) \leqs 0$.

	(i) $\Rightarrow$ (ii). The function $f$ is operator convex on the restriction $(0,\ap)$.
        So, it is convex on $(0,\ap)$ which implies that $f(0^+)$ exists and $f(0^+) \leqs f(0)$.

	(ii) $\Rightarrow$ (i). Define a function $f_0$ on $[0,\ap)$ by $f_0 (0) =f(0^+)$
    and $f_0(t) = f(t)$ for $t>0$.
	Then $f_0$ is continuous since it is convex on the open set $(0,\ap)$.
	The continuity argument shows that $f_0$ is operator convex on $[0,\ap)$.
	Now consider $A$ and $P$ as in (vi).
    Let $Q_0$ be the orthogonal projection onto the kernel of $A$
    and $\tilde{Q}_0$ be that on the kernel of $PAP$.
    Then $Q_1 := I - Q$ is the orthogonal projection onto the range of
    $A$ and $\tilde{Q_1} := I - \tilde{Q}_0$ is the orthogonal projection onto the range of $PAP$.
    It follows that
    \begin{align*}
        f(PAP) &= f_0(PAP) + \ap \tilde{Q}_0, \\
        Pf(A)P &= P(f_0 (A) + \ap Q_0)P = P f_0(A) P + \ap P Q_0 P,
    \end{align*}
	here $\ap:= f(0) - f(0^+) \geqs 0$.
    By applying the implication (i) $\Rightarrow$ (ii) to $f_0$, we have $f_0(PAP) \leqs Pf_0(A)P$.
    Using the orthogonal decomposition $\C^n = P\C^n \oplus (I-P)\C^n$, we have
    \begin{align*}
        f(PAP) &= P f(PAP) P + f(0)(I-P)  \leqs Pf(PAP)P \\
         &= P f_0(PAP)P + \ap P \tilde{Q}_0 P \leqs P f_0 (A) P + \ap P \tilde{Q}_0 P.
    \end{align*}
    We will show that $P \tilde{Q}_0 P \leqs P Q_0 P$ which implies $f(PAP) \leqs P f(A) P$
    and hence (i) holds.
    Choose $\delta >0$ such that $A \geqs \delta Q_1$ and $\tilde{Q}_1 \geqs \delta PAP$.
    Then $\tilde{Q}_1 \geqs \delta^2 P Q_1 P$ and $(I-\tilde{Q}_1) PQ_1 P(I- \tilde{Q}_1) = 0$.
    Hence $Q_1 P(I - \tilde{Q}_1) = 0$ and $P Q_1 P = \tilde{Q}_1 P Q_1 P \tilde{Q}_1 \leqs \tilde{Q}_1$.
    Thus $P Q_1 P \leqs P \tilde{Q}_1 P$ or $P Q_0 P \geqs P \tilde{Q}_0 P$.

	(iv) $\Rightarrow$ (iii).
	Let $A \geqs B >0$. Setting $X = A^{-1/2} B^{1/2}$,
	we have $XX^* = A^{-1/2} B A^{-1/2} \leqs I$, i.e. $\norm{X} \leqs 1$.
	Since $B = X^*AX$, (iv) implies that
	\begin{align*}
		f(B) \leqs X^* f(A) X = B^{1/2} A^{-1/2} f(A) A^{-1/2} B^{1/2}
	\end{align*}
	and $A^{-1} f(A) = A^{-1/2}f(A) A^{-1/2} \geqs B^{-1/2} f(B) B^{-1/2} = B^{-1} f(B)$.
	Thus, $f(t)/t$ is operator monotone on $(0,\ap)$.

    (iii) $\Rightarrow$ (ii). First we prove that if $g$ is a continuous operator monotone function on
	$[0,\ap)$, then $h(t):=tg(t)$ is operator convex on $[0, \ap)$.
	To prove (vi) for $h$, we may assume that $A>0$.
	Since $A^{1/2} P A^{1/2} \leqs A$, we have $g(A^{1/2} P A^{1/2}) \leqs g(A)$.
	Then
	\begin{align*}
		PA^{1/2} g(A^{1/2} P A^{1/2}) A^{1/2} P \leqs PA^{1/2} g(A) A^{1/2} P.
	\end{align*}
	By Lemma \ref{Lemma 2.5.1}(2), we have $g(A^{1/2} P A^{1/2}) A^{1/2} P = A^{1/2} P g(PAP)$. Hence,
	\begin{align*}
		h(PAP) &= PAP g(PAP) = P A^{1/2} A^{1/2} P g(PAP)
			= PA^{1/2} g (A^{1/2} P A^{1/2}) A^{1/2} P  \\
		 &\leqs PA^{1/2} g(A) A^{1/2} P = PAg(A)P = P h(A) P,
	\end{align*}
	i.e. $h$ is operator convex on $[0,\ap)$ and the claim follows.

	Now, assume that $f(t)/t$ is operator monotone on $(0,\ap)$.
	By Theorem \ref{Theorem 2.4.1}, $f(t)/t$ is continuous on $(0,\ap)$.
	For each $\ep>0$, $f(t+ \ep) / (t + \ep)$ is continuous and operator monotone on $[0, \ap-\ep)$.
	The previous claim implies that $\frac{t}{t+\ep} f(t + \ep)$ is operator convex on $[0,\ap-\ep)$.
	Hence, by letting $\ep \searrow 0$, $f$ is operator convex on $(0,\ap)$.
\end{proof}

\begin{thm} \label{thm 2.5.3}
	If $\ap = \infty$ and $f(t) \leqs 0$ for all $t \in [0,\infty)$, then the conditions of Theorem \ref{thm 2.5.2}
	is also equivalent to
	
	(vii)	$-f$ is operator monotone on $[0,\infty)$.
\end{thm}
\begin{proof}
Assume that $f \leqs 0$ on $[0,\infty)$. First we prove that (vii) is equivalent to

(viii) $-f$ is operator monotone on $(0,\infty)$ and $f(0^+) \leqs f(0)$.

(vii) $\Rightarrow$ (viii). If (vii) holds, then $f(0^+)$ exists and $f(0^+) \leqs f(0)$.

(viii) $\Rightarrow$ (vii). Define $f_0$ on $[0,\infty)$ by $f_0(0) = f(0^+)$ and $f_0(t) = f(t)$ for $t>0$.
				Then $f_0$ is continuous on $[0,\infty)$.
				Hence $-f_0$ is operator monotone on $[0,\infty)$.
				Consider $A \geqs B \geqs 0$. Let $Q_0$ and $\tilde{Q}_0$ be the projection onto
                the kernels of $A$ and $B$, respectively.
				Then
				\begin{align*}
					f(A) = f_0(A) + \ap Q_0, \quad f(B) = f_0(B) + \ap \tilde{Q}_0,
				\end{align*}
				where $\ap :=f(0) - f(0^+) \geqs 0$. Since $A \geqs B \geqs 0$, we have $Q_0 \leqs \tilde{Q}_0$.
				With $f_0(A) \leqs f_0(B)$ this implies that $f(A) \leqs f(B)$.
				Thus, it suffices to prove that (i) $\Leftrightarrow$ (vii) for the function $f_0$.
Since $f_0$ is cont. on $[0,\infty)$, we can assume that $f$ is continuous.
			
	(vii) $\Rightarrow$ (iv). Define
	\begin{align*}
				\tilde{A} = \begin{bmatrix} A & 0 \\ 0 & 0 \end{bmatrix}, \;
								U = \begin{bmatrix} X & R \\ (I-X^*X)^{1/2}  & -X^* \end{bmatrix}, \;
				\tilde{B} = \begin{bmatrix} X^*AX + \ep I & 0 \\ 0  &  \bt I \end{bmatrix}
	\end{align*}
	for each $\ep,\bt>0$ where $R:= (I-XX^*)^{1/2}$.
	The computation shows that $\tilde{B} - U^* \tilde{A} U \geqs 0$ for sufficient large $\bt$.
	Then (vii) implies that
	\begin{align*}
		\begin{bmatrix}	f(X^*AX + \ep I) & 0 \\ 0 & f(\bt)I  \end{bmatrix}
			&=	f(\tilde{B})	\leqs	f(U^* \tilde{A} U)	=	U^* \begin{bmatrix}	f(A) & 0 \\ 0 & f(0)I
                \end{bmatrix} U \\
			&\leqs	U^* \begin{bmatrix}	f(A) & 0 \\ 0 & 0  \end{bmatrix} U	= \begin{bmatrix}	X^*f(A)X & * \\ * & *
        \end{bmatrix}.
	\end{align*}
	Hence	$f(X^*AX + \ep I) \leqs X^* f(A) X$. Letting $\ep \searrow 0$ yields $f(X^*AX) \leqs X^* f(A) X$.
	
	(i) $\Rightarrow$ (vii).	Consider $A \geqs B \geqs 0$.
	For each $0< \ld <1$, since $\ld A = \ld B + (1-\ld) \ld (1- \ld)^{-1} (A-B)$, we have
	\begin{align*}
		f(\ld A)	\leqs \ld f(B) + (1- \ld) f(\ld (1-\ld)^{-1} (A-B)) \leqs \ld f(B).
	\end{align*}
	Letting $\ld \nearrow 1$ yields $f(A) \leqs f(B)$, meaning that $-f$ is operator monotone on $[0,\infty)$.
\end{proof}

\begin{cor} \label{Cor 2.5.4}
		A function on $f:[0,\infty) \to [0,\infty)$ is operator monotone if and only if
		$f$ is operator concave.
\end{cor}
\begin{proof}
	This is the equivalence between (i) and (vii) of Theorem \ref{thm 2.5.3}.
\end{proof}

\begin{cor}	 \label{Cor 2.5.6}
	Consider the following statements for a function $f : (0,\infty) \to (0,\infty)$.
	\begin{enumerate}
			\item[(i)]	$f$ is operator monotone;
			\item[(ii)]	$t/f(t)$ is operator monotone;
			\item[(iii)]	$f$ is operator concave;
			\item[(iv)]	$1/f(t)$ is operator convex.
	\end{enumerate}
	We have (i) $\Leftrightarrow$ (ii) $\Leftrightarrow$ (iii) $\Rightarrow$ (iv).
\end{cor}
\begin{proof}
(i) $\Rightarrow$ (ii).	
		For any $\ep>0$, $f(t+\ep)$ is operator monotone on $[0,\infty)$. 
		Theorem	\ref{thm 2.5.3} implies that $-f(t+\ep)/t$ is operator monotone on $(0,\infty)$.
		Proposition \ref{Lemma 2.5.5} then implies that
		\begin{align*}
			\frac{t}{f(t+\ep)} = -\left(-\frac{f(t + \ep)}{t} \right)^{-1}
		\end{align*}
		is operator monotone on $(0,\infty)$. Letting $\ep \searrow 0$ yields (ii).

(ii) $\Rightarrow$ (i).	
		For any $\ep>0$, $(t+\ep)/f(t+\ep)$ is operator monotone on $[0,\infty)$. 
		Theorem	\ref{thm 2.5.3} implies that $-(t+\ep)/tf(t+\ep)$ is operator monotone on $(0,\infty)$.
		Proposition \ref{Lemma 2.5.5} then implies that
		\begin{align*}
			\frac{t f(t+\ep)}{t+\ep} = -\left(-\frac{t + \ep}{t f(t+\ep)} \right)^{-1}
		\end{align*}
		is operator monotone on $(0,\infty)$. Letting $\ep \searrow 0$ yields (i).	

(i) $\Leftrightarrow$ (iii). By Corollary \ref{Cor 2.5.4}, we have that
			\begin{align*}
				f &\text{ is operator monotone on } (0,\infty)  \\
					&\Leftrightarrow f(t+\ep) \text{ is operator montone on } [0,\infty) \text{ for any } \ep>0 \\
					&\Leftrightarrow f(t+\ep) \text{ is operator concave on } [0,\infty) \text{ for any } \ep>0 \\
					&\Leftrightarrow	f \text{ is operator concave on } (0,\infty)
			\end{align*}
	(iii) $\Leftrightarrow$ (iv). Write $g(t) = 1/f(t)$. Let $A,B>0$ in $M_n$. By (iii),
	\begin{align*}
		f \left( \frac{A+B}{2}\right) \geqs \frac{f(A)+f(B)}{2}.
	\end{align*}
	Then Proposition \ref{Lemma 2.5.5} implies
	\begin{align*}
		g \left( \frac{A+B}{2}\right)
			&= f \left( \frac{A+B}{2}\right)^{-1} \leqs \left\{\frac{f(A)+f(B)}{2} \right\}^{-1} \\
			&\leqs	\frac{f(A)^{-1} + f(B)^{-1}}{2} = \frac{g(A)+g(B)}{2}.
	\end{align*}
	Hence $g$ is operator convex.
\end{proof}

\begin{ex}
    \begin{enumerate}
        \item[(i)]   For each $p \in [0,1]$, $t^p$ is operator concave on $[0,\infty)$.
        \item[(ii)]   The function $f(t) =(t-1)/\log t$ on $[0,\infty)$ where $f(0):= 0$ and $f(1):=1$.
        \item[(iii)]   The logarithmic function is operator monotone and operator concave on $(0,\infty)$.
        \item[(iv)]   The function $g(t) = t \log t$ is operator convex on $[0,\infty)$.
    \end{enumerate}
\end{ex}
\begin{proof}
    (i) It follows from the L\"{o}wner-Heinz inequality and Corollary \ref{Cor 2.5.4}.

    (ii) Note that $f(t) = \int_0^1 t^x \, dx$ for $t \geqs 0$.

    (iii) By (ii), $t/\log (1+t)$ is operator monotone function on $(0,\infty)$.
          Corollary \ref{Cor 2.5.6} then implies that $\log (1+t)$ is operator monotone and operator concave on $(0,\infty)$. Now, for each $\ep>0$, $\log (\ep+t) = \log \ep + \log(1+\ep^{-1}t)$ is operator monotone
          and operator concave on $(0,\infty)$. Letting $\ep \searrow 0$ yields the result.

    (iv)   Since $g$ is continuous on $[0,\infty)$ and $g(t)/t = \log t$ is operator monotone on $(0,\infty)$,
           $g$ is operator convex on $[0,\infty)$ by Theorem \ref{thm 2.5.2} .
\end{proof}

\section{Integral representations of operator monotone functions on the nonnegative reals}


The aim of this section is to show that every operator monotone function from $\R^+$ to itself
always arises as an integral of special operator monotone functions with respect to a Borel measure:

\begin{thm} \label{Theorem 2.7.6}
	A continuous function $f: [0,\infty) \to [0,\infty)$ is operator monotone if and only if
	there is a finite Borel measure $m$ on $[0,\infty]$ such that
	\begin{align}
		f(t) = \int_{[0,\infty]} \phi_t (\ld) \,dm(\ld), \quad t \in [0,\infty)
        \label{eq: int rep of operat mon on R+}
	\end{align}
	where
	\begin{align*}
		\phi_t (\ld) = \frac{t(1+ \ld)}{t+\ld} \; for \; \ld \in (0,\infty),
		\quad \phi_t (0)=1, \quad \phi_t (\infty) =t.
	\end{align*}
	Moreover, the measure $m$ is unique and we can write
	\begin{align*}
		f(t) = a+bt+\int_{(0,\infty)} \frac{t(1+ \ld)}{t+\ld} \,dm(\ld), \quad t \in [0,\infty)
	\end{align*}
	where $a:= m(\{0\}) = f(0)$ and $b:=m(\{\infty\}) = \lim_{t \to \infty} f(t)/t$.
\end{thm}

Hence there is a one-to-one correspondence between operator monotone functions on the nonnegative reals
and finite Borel measures on the extended half-line. The operator monotone functions $t \mapsto \phi_t (\ld)$
for each fixed $\ld \in [-1,1]$ form a building block for constructing general operator monotone functions on the
nonnegative reals. It follows immediately that the map $f \mapsto m$ is affine.

In order to prove Theorem \ref{Theorem 2.7.6}, we use the following theorem.

\begin{thm}[Krein-Milman]
		A convex compact subset of a locally convex topological vector space always has an extreme point.
Moreover, it is the closed convex hull of the set of its extreme points.
\end{thm}

Operator monotone functions on $(a,b)$ are transformed to those on a symmetric interval $(-1,1)$
via an affine function which is also operator monotone.
Recall that every operator monotone function on $(-1,1)$ is $C^1$ and $f'>0$ unless $f$ is a constant.
Denote by $\K$ the set of operator monotone functions $f$ on $(-1,1)$ such that $f(0)=0$ and $f'(0)=1$.
It is easy to see that $\K$ is convex. The next three lemmas establish that $\K$ is a compact subset
of the locally convex space of real-valued functions on $(-1,1)$.

\begin{lem} \label{Lemma 2.7.1}
	Let $f$ be an operator monotone function on $(-1,1)$.
	\begin{enumerate}
		\item[(1)]	Then for every $\ap \in [-1,1]$, $(x + \ap ) f(x)$ is operator convex on $(-1,1)$.
		\item[(2)]	If $f(0)=0$, then for every $\ap \in [-1,1]$, $g(x)=(1 + \frac{\ap}{x} ) f(x)$
                    is operator monotone on $(-1,1)$. Here, $g(0):= \lim_{x \to 0} g(x) = \ap f'(0)$.
		\item[(3)]	If $f(0)=0$, then $f$ is twice differentiable at $0$ and
					\begin{align*}
						\frac{f''(0)}{2} = \lim_{x \to 0} \frac{f(x) - f'(0)x}{x^2}.
					\end{align*}
	\end{enumerate}
\end{lem}
\begin{proof}
	$(1)$
	Let $\ap \in [-1,1]$. Note that
	\begin{align*}
		(x+\ap) f(x) = \frac{1+\ap}{2} (x+1)f(x) + \frac{1-\ap}{2} (x-1)f(x).
	\end{align*}
	For each $\ep \in (0,1)$, $f(x+1-\ep)$ is operator monotone on $(-\ep, 2-\ep)$
	and hence on $[0,2-\ep)$.
	Theorem \ref{thm 2.5.2} implies that $xf(x-1+\ep)$ is operator convex on $[0,2-\ep)$.
	Hence $(x+1-\ep) f(x)$ is operator convex on $(-1+\ep,1)$.
	Letting $\ep \searrow 0$ yields that $(x+1)f(x)$ is operator convex on $(-1,1)$.
	Similarly, $(x-1)f(x)$ is operator convex on $(-1,1)$.

	$(2)$ For each $\ap \in [-1,1]$, set $g(x) = (x+ \ap) f(x)$.
	By (1) and Theorem \ref{thm 2.4.4 and cor 2.4.6},
	\begin{align*}
		g^{[1]} (0,x) = \frac{g(0)-g(x)}{0-x} = \frac{g(x)}{x} = (1+ \frac{\ap}{x})f(x),
		 \quad x \ne 0
	\end{align*}
	is operator monotone on $(-1,1)$.
	
	$(3)$ By (2), $(1+ \frac{1}{x}) f(x)$ and $f(x)$
	are $C^1$ on $(-1,1)$. Define $h: (-1,1) \to \R$ by
	\begin{align*}
		h(x) = \begin{cases} f(x)/x, & x \ne 0 \\ f'(0) & x=0.\end{cases}
	\end{align*}
	Then $h$ is $C^1$. An elementary calculation shows that
	\begin{align*}
		\lim_{x \to 0} \frac{f'(x) - f'(0)}{x-0} = 2h'(0)
		= 2 \lim_{x \to 0} \frac{f(x) - f'(0)x}{x^2}.
	\end{align*}
\end{proof}

\begin{lem} \label{Lemma 2.7.2}
	If $f \in \K$, then
	\begin{align}
		f(x) \quad &\leqs \quad \frac{x}{1-x}, \quad 0 \leqs x <1,    \label{eq: 1}\\
		f(x) \quad &\geqs \quad \frac{x}{1+x}, \quad -1 < x \leqs 0,  \label{eq: 2}\\
		|f''(0)| \quad &\leqs \quad 2     \label{eq: 3}.
	\end{align}
\end{lem}
\begin{proof}
For each $x \in (-1,1)$, since $f$ is $2$-monotone, we have
\begin{align*}
	\begin{bmatrix}
		f^{[1]}(x,x) & f^{[1]}(x,0) \\ f^{[1]}(0,x) & f^{[1]}(0,0)\end{bmatrix}
		=
	\begin{bmatrix} f'(x) & f(x)/x \\ f(x)/x & 1 \end{bmatrix}
	\geqs 0,
\end{align*}
and hence
	\begin{align}
		f(x)^2 /x^2 \leqs f'(x).  \label{eq: 2.7.1}
	\end{align}
	By Lemma \ref{Lemma 2.7.1}(1), $g(x):=(x \pm 1)f(x)$ is operator convex on $(-1,1)$.
	Theorem  \ref{thm 2.4.4 and cor 2.4.6} implies that
	\begin{align*}
		g'(x) = \lim_{y \to x} \frac{g(y)-g(x)}{y-x} = \lim_{y \to x} g^{[1]}(y,x),
		\quad x \in (-1,1)
	\end{align*}
	is operator monotone on $(-1,1)$. In particular, it is increasing on $(-1,1)$.
	This implies that
	\begin{align}
		f(x) + (x-1)f'(x) &\geqs -1, \quad 0<x<1,     \label{eq: 2.7.2}\\
		f(x) + (x+1)f'(x) &\leqs -1, \quad -1<x<0.    \label{eq: 2.7.3}
	\end{align}
	From \eqref{eq: 2.7.1} and \eqref{eq: 2.7.2}, we obtain $f(x)+1 \geqs (1-x)f(x)^2 / x^2$.
    If $f(x)> x/(x-1)$ for some $x \in (0,1)$, then
    \begin{align*}
        f(x)+1 > \frac{(1-x)f(x)^2}{x^2} \cdot \frac{x}{1-x} = \frac{f(x)}{x}
    \end{align*}
	so that $f(x) < x/(x-1)$. Hence $f(x) \leqs x/(1-x)$ for all $x \in [0,1)$.
    Similarly, using \eqref{eq: 2.7.1} and \eqref{eq: 2.7.3}, $f(x) \geqs x/(1-x)$ for all $x \in (-1,0]$.
	To prove \eqref{eq: 3}, use Lemma \ref{Lemma 2.7.2} and the inequalities \eqref{eq: 1}
	and \eqref{eq: 2}.
\end{proof}

\begin{lem} \label{Lemma 2.7.3}
	The set $\K$ is compact if it is considered as a subset of a topological vector space
	consisting of real functions on $(-1,1)$ with the locally convex topology of pointwise convergence.
\end{lem}
\begin{proof}
    Recall that the space $\R^{(-1,1)}$ of functions from $(-1,1)$ to $\R$ is homeomorphic to the
    space $\prod_{x \in (-1,1)} \R$ with product topology.
    By Lemma \ref{Lemma 2.7.2}, the set $A_x:=\{f(x): f \in \K\}$ is bounded for each $x \in (-1,1)$.
    Then
    \begin{align*}\overline{\prod_{x \in (-1,1)} A_x} = \prod_{x \in (-1,1)} \overline{A_x} \end{align*}
    is compact by Tychonoff's theorem.
    We will show that $\K$ is closed in $ \prod_{x \in (-1,1) } \R \cong \R^{(-1,1)}$.
    To show that $\K$ is closed in $\R^{(-1,1)}$, let $\{f_i\}$ be a net in $\K$ converging to a function $f$ on $(-1,1)$. It is clear that $f$ is operator monotone on $(-1,1)$ and $f(0)=0$.
    Lemma \ref{Lemma 2.7.1} implies that $(1+\frac{1}{x})f_i (x)$ is operator monotone on $(-1,1)$ for every $i$.
    Then Lemma \ref{Lemma 2.7.2} and the fact that $\lim_{x \to 0} (1+\frac{1}{x}) f_i(x) = f_i'(0)=1$ yield
    that for each $i$
    \begin{align*}
        (1 - \frac{1}{x}) f_i (-x) \leqs 1 \leqs (1 + \frac{1}{x}) f_i (x), \quad x \in (0,1).
    \end{align*}
    By continuity,
    \begin{align*}
        (1 - \frac{1}{x}) f (-x) \leqs 1 &\leqs (1 + \frac{1}{x}) f (x), \quad x \in (0,1), \\
        (1 + \frac{1}{x}) f ( x) \leqs 1 &\leqs (1 - \frac{1}{x}) f (-x), \quad x \in (-1,0).
    \end{align*}
    Now,
    \begin{align*}
        \lim_{x \to 0^+} \frac{f(x)-f(0)}{x-0} &= \lim_{x \to 0^+} (1+\frac{1}{x})f(x) \geqs 1 \\
        \lim_{x \to 0^-} \frac{f(x)-f(0)}{x-0} &= \lim_{x \to 0^-} (1+\frac{1}{x})f(x) \leqs 1.
    \end{align*}
    Since $f$ is $C^1$ on $(-1,1)$, this forces $f'(0)=1$, i.e. $f \in \K$.
\end{proof}

\begin{lem} \label{Lemma 2.7.4}
	The extreme points of $\K$ are of the form
	\begin{align*}
		f(x) = \frac{x}{1- \ld x}, \quad where \quad \ld \in [-1,1].  
	\end{align*}
\end{lem}
\begin{proof}
    Let $f$ be an extreme point of $\K$. For each $\ap \in (-1,1)$, define
    \begin{align*}
        g_{\ap}(x) = (1+ \frac{\ap}{x})f(x) - \ap, \quad x \in (-1,1).
    \end{align*}
    By Lemma \ref{Lemma 2.7.1}(2), $g_{\ap}$ is operator monotone.
    Note that $g_{\ap}(0) := \lim_{x \to 0} g_{\ap}(x) = 0$.
    By Lemma \ref{Lemma 2.7.1}(3) $g_{\ap}'(0) = 1+\frac{1}{2}f''(0)$.
    Since $g_{\ap}'(0) > 0$ by Lemma \ref{Lemma 2.7.2}, the function
    $
        h_{\ap}(x) := g_{\ap}(x)/g'_{\ap}(0)
    $, $x \in (-1,1)$,
    belongs to $\K$.
    Since $f$ can be written as a convex combination of two elements in $\K$, namely,
    \begin{align*}
        f = \frac{1}{2} (1+ \frac{1}{2} \ap f''(0)) h_{\ap} + \frac{1}{2} (1 - \frac{1}{2} \ap f''(0)) h_{-\ap},
    \end{align*}
    the extremality of $f$ implies that $f = h_{\ap}$. We can solve for $f$ so that $f(x)=x/ (1-\frac{1}{2}f''(0)x)$.
    Since $|f''(0)| \leqs 2$, $\ld:=f''(0)$ can be varied in $[-1,1]$.
\end{proof}

\begin{thm} \label{Theorem 2.7.5}
	Let $f$ be a non-constant operator monotone function on $(-1,1)$. Then there is a unique probability Borel measure $\mu$ on $[-1,1]$ such that
	\begin{align*}
		f(x) = f(0) + f'(0) \int_{-1}^{1} \frac{x}{1- \ld x} \,d\mu(\ld), \quad x \in (-1,1).
	\end{align*}
\end{thm}
\begin{proof}
    Since $f'>0$, by considering $(f-f(0))/f'(0)$ we can assume that $f \in \K$.
    For each $\ld \in [-1,1]$, let $\phi_{\ld}(x) = x/(1-\ld x)$, $x \in (-1,1)$.
    Lemmas \ref{Lemma 2.7.3} and \ref{Lemma 2.7.4} mean that $\K$ is convex and compact and the extreme points of
    $\K$ are of the form $\phi_{\ld}$ for some $\ld \in [-1,1]$.
    The Krein-Milman theorem says that $\K$ is the closure of the convex hull $\E$ of $\{\phi_{\ld}: \ld \in [-1,1]\}$.
    Let $\{f_i\}$ be a net in $\E$ such that $f_i(x) \to f(x)$ for all $x \in (-1,1)$.
    Each $f_i$ can be written as
    \begin{align*}
        f_i (x) = \int_{[-1,1]} \phi_{\ld} (x)\, d\mu_i (\ld), \quad x \in (-1,1)
    \end{align*}
    with a probability measure $\mu_i$ on $[-1,1]$ with finite support.
    Recall the Riesz-representation theorem that the Banach space $M[-1,1]$ of complex Borel measures on $[-1,1]$
    is the dual space of the space $C[-1,1]$
    of complex-valued continuous functions on $[-1,1]$.
    The Banach space $M_1 [-1,1]$ of probability Borel measures on $[-1,1]$ is compact
    in the weak$^*$ topology when considered as a subset of $M[-1,1]$ by Banach-Alaoglu theorem.
    Since $\K$ is compact, by taking a subnet we may assume that the net $\mu_i$ converges in the weak$^*$ topology
    to a net $\mu \in M_1 [-1,1]$. For each $x \in (-1,1)$, since $\ld \mapsto \phi_{\ld}(x)$ is continuous,
    we have
    \begin{align*}
	f(x) = \lim_i f_i(x) = \lim_i \int_{-1}^1 \phi_{\ld} (x) \,d\mu_i(\ld) =  \int_{-1}^1 \phi_{\ld} (x) \,d\mu(\ld).
    \end{align*}
    To prove the uniqueness of the measure, let $\mu_1, \mu_2 \in M_1 [-1,1]$ be such that
\begin{align*}
	\lim_i \int_{-1}^1 \phi_{\ld} (x) \,d\mu_1(\ld) = f(x) = \lim_i \int_{-1}^1 \phi_{\ld} (x) \,d\mu_2(\ld),
	\quad x \in (-1,1).
\end{align*}
	Note that $\phi_{\ld}(x) = \sum_{k=0}^{\infty} x^{k+1}\ld^k$ is uniformly convergent in $\ld \in [-1,1]$
	for any $x \in (-1,1)$ fixed. Then
	\begin{align*}
		\sum_{k=0}^{\infty} x^{k+1} \int_{-1}^1 \ld^k \,d\mu_1(\ld)
		= \sum_{k=0}^{\infty} x^{k+1} \int_{-1}^1 \ld^k \,d\mu_2(\ld), \quad x \in (-1,1).
	\end{align*}
	Hence, $\int_{-1}^1 \ld^k\, d\mu_1(\ld) = \int_{-1}^1 \ld^k\, d\mu_2(\ld)$ for all $k=0,1,2,\dots$.
	Thus, $\int_{-1}^1 p(\ld)\, d\mu_1(\ld) = \int_{-1}^1 p(\ld)\, d\mu_2(\ld)$ for all polynomials $p$ on $[-1,1]$.
	The Stone-Weierstrass theorem implies that
	\begin{align*}
		\int_{-1}^1 f(\ld)\, d\mu_1(\ld) = \int_{-1}^1 f(\ld)\, d\mu_2(\ld)
	\end{align*}
	for all $f \in C[-1,1]$.
	This implies $\mu_1 = \mu_2$ by the Riesz Representation Theorem.
\end{proof}

\begin{proof}[Proof of Theorem \ref{Theorem 2.7.6}]
    ($\Leftarrow$) For each $\ld \in [0,\infty)$,
	\begin{align*}
		t \mapsto \frac{t(1+\ld)}{t+\ld} = 1+\ld- \frac{\ld (1+\ld)}{t+\ld}
	\end{align*}
	is operator monotone function on $[0,\infty)$. It follows that if $A \geqs B$ in $M_n^+$, we have
    \begin{align*}
        f(A) = \int_{[0,\infty]} (1+\ld)A(A+\ld I)^{-1} \, d\mu(\ld)
        \geqs \int_{[0,\infty]} (1+\ld)B(B+\ld I)^{-1} = f(B).
    \end{align*}
	\noindent ($\Rightarrow$) Assume that $f$ is operator monotone
	on $[0,\infty)$. Transform $f(t)$ on $(0,\infty)$ to an operator monotone function
	$g(x) := f(\psi (x))$ on $(-1,1)$ by
	\begin{align*}
		t = \psi(x) = \frac{1+x}{1-x} = -1 + \frac{2}{1-x}: (-1,1) \to (0,\infty).
	\end{align*}
	Theorem \ref{Theorem 2.7.6} implies that there is a probability Borel measure $\mu$ on $(-1,1)$ such that
	\begin{align*}
		g(x) = g(0) + g'(0)\int_{[-1,1]} \frac{x}{1-\ld x} \,d\mu(\ld),
		\quad x \in (-1,1).
	\end{align*}
	Then
	\begin{align*}
		0 \leqs f(x) = g(-1) &= g(0)+g'(0)\lim_{x \to -1^+}
		\int_{[-1,1]} \frac{x}{1-\ld x} \,d\mu{\ld}  \\
		&= - \int_{[-1,1]} \frac{1}{1+ \ld} \,d\mu{\ld}
	\end{align*}
	and in particular $\mu(\{-1\})=0$.
	Hence
	\begin{align*}
		g(x)-g(-1) = g'(0)\int_{(-1,1]} \frac{1+x}{(1-\ld x)(1+ \ld)} \,d\mu{\ld}.
	\end{align*}
	Transform this to the term of $f(t)$ by $x=\psi^{-1}(t)$ and $\ld= \psi^{-1}(\xi)$
	and introducing the measure $\mu$ on $(0,\infty]$ by
	\begin{align*}
		m:= \tilde{\mu} \circ \psi^{-1}, \quad \text{ where } \quad
		\tilde{\mu}(\ld) := \frac{g'(0)}{1+\ld} \,d\mu(\ld).
	\end{align*}
	We now have
	\begin{align*}
		f(t)-f(0) = \int_{(0,\infty]} \frac{t(1+\xi)}{t+\xi} \,dm(\xi),
		\quad t \in [0,\infty)
	\end{align*}
	and hence
	\begin{align*}
		f(t) = \int_{[0,\infty]} \frac{t(1+\xi)}{t+\xi} \,dm(\xi),
		\quad t \in [0,\infty)
	\end{align*}
	The uniqueness of the measure $m$ follows from that of $\mu$ in Theorem \ref{Theorem 2.7.6}.
\end{proof}

\begin{ex}
    \begin{enumerate}
        \item   The function $t \mapsto 1$ is associated to the Dirac measure $\delta_0$ at $0$.
        \item   The function $t \mapsto t$ is associated to the Dirac measure $\delta_{\infty}$.
        \item   Recall that the operator monotone function $t^p$ for $0<p<1$ has an integral representation
                \begin{align*}
                    t^p = \frac{\sin p\pi}{\pi} \int_{[0,\infty]} \frac{t}{\ld(t + \ld)}\,d\mu(\ld).
                \end{align*}
                Hence the representing measure of $t^p$ is $\frac{\sin p\pi}{\pi} \cdot \frac{\ld^{p-1}}{1+\ld} \,d \ld$.
    \end{enumerate}
\end{ex}

\end{document}